\documentclass{amsart}
\usepackage{amssymb}
\usepackage{amsthm}
\usepackage{pb-diagram}
\usepackage{color}

\usepackage[textsize=tiny]{todonotes}

\newtheorem{theorem}{Theorem}[section]

\newtheorem{prop}[theorem]{Proposition}

\newtheorem{claim}[theorem]{Claim}
\newtheorem{fact}[theorem]{Fact}

\newtheorem{proposition}[theorem]{Proposition}
\newtheorem{lemma}[theorem]{Lemma}

\newtheorem{question}[theorem]{Question}

\newtheorem*{question*}{Question}

\theoremstyle{definition}
\newtheorem{definition}[theorem]{Definition}

\newtheorem{question+}[theorem]{Question}

\newtheorem{defn}[theorem]{Definition}

\theoremstyle{remark}
\newtheorem{rmk}[theorem]{Remark}
\newtheorem{remark}[theorem]{Remark}

\newcommand{\WI}{\widetilde{\cal I}}
\newcommand{\WR}{\widetilde{\cal R}}

\newcommand{\la}{\langle}
\newcommand{\ra}{\rangle}

\newcommand{\sub}{\subseteq}

\newcommand{\CI}{{\cal I}}

\newcommand{\fr}{\ensuremath{\textup{fr}}}

\newcommand{\cal}[1]{\ensuremath{\mathcal{#1}}}

\newcommand{\Lrarr}{\ensuremath{\Leftrightarrow}}

\newcommand{\res}{\ensuremath{\upharpoonright}}

\newcommand{\cl}[1]{\begin{overline}{#1}\end{overline}}

\newcommand{\ve}{\ensuremath{\varepsilon}}

\newcommand{\lam}{\ensuremath{\lambda}}

\newcommand{\es}{\ensuremath{\emptyset}}

\newcommand{\dom}{\ensuremath{\textup{dom}}}

\newcommand{\sm}{\setminus}

\newcommand{\Z}{\mathbb{Z}}
\newcommand{\N}{\mathbb{N}}

\newcommand{\R}{\mathbb{R}}

\title[On semibounded expansions of ordered groups]
{On semibounded expansions of ordered groups}

\begin{document}

\author {Pantelis  E. Eleftheriou}

\address{School of Mathematics, University of Leeds, Leeds LS2 9JT, United Kingdom}

\email{p.eleftheriou@leeds.ac.uk}

\author {Alex Savatovsky}

\address{Department of Mathematics, University of Haifa, Haifa, Israel}

\email{alex.savatovsky@uni-konstanz.de}

\thanks{The first author was partially supported by an EPSRC Early Career Fellowship (EP/V003291/1) and a Zukunftskolleg Research Fellowship (Konstanz). The second author was partially supported by the Israel Science Foundation (Grant 290/19).}


\date{\today}
\begin{abstract} We explore \emph{semibounded} expansions of arbitrary ordered groups; namely, expansions that do not define a field on the whole universe. 
We show that if $\cal R=\la \R, <, +, \dots\ra$ is a semibounded o-minimal structure and $P\sub \R$ is a set satisfying certain tameness conditions, then $\la \cal R, P\ra$ remains semibounded. Examples include the cases when $\cal R=\la\R,<,+, (x\mapsto \lam x)_{\lam \in \R}, \cdot_{\res [0, 1]^2}\ra$, and $P= 2^\Z$ or $P$ is an iteration sequence. 
As an application, we obtain that smooth functions definable in such $\la \cal R, P\ra$ are definable in $\cal R$.

\end{abstract}
 \maketitle

\section{Introduction}

The work of this paper lies at the nexus of two different directions in model theory, both related to o-minimality, which  so far have developed independently. The first direction is that of o-minimal \emph{semibounded} structures, which are o-minimal structures that do not interpret a global field, and are obtained, for example, as proper \emph{reducts} of real closed fields. These structures were extensively studied in the 90s by Marker, Peterzil, Pillay \cite{mpp, pet-reals, pss}  and others, they relate to 
Zilber's dichotomy principle on definable groups and fields,
and have continued to develop in recent years \cite{ed-sbd, el-sbd, pet-sbd}.

The second direction is that of \emph{expansions} $\WR$ of o-minimal structures \cal R which are not o-minimal, yet preserve the tame geometric behavior on the class of all definable sets. This area is much richer,  originating to A. Robinson \cite{rob} and van den Dries \cite{VDD1, vdd-dense}, it has largely expanded in the last two decades by many authors, and  includes broad categories of structures, such as d-minimal  expansions of o-minimal structures and expansions  with o-minimal open core. Although in general $\cal R$ is only required to expand a linear order, it is often assumed to expand an ordered group or even a real closed field (and, in fact,  the real field $\overline \R$).

In recent work \cite{HW}, Hieronymi-Walsberg considered expansions of ordered groups and explored the dichotomy between defining or not a local field. In this paper, we consider expansions of ordered groups  and explore the dichotomy between defining or not a global field (in the latter case, call $\WR$ \emph{semibounded}). As an application, and building on the work from \cite{es}, we obtain that for certain semibounded expansions $\WR$ of o-minimal structures $\cal R$, such as $\WR=\la \cal R, 2^\Z\ra$ with $\cal R=\la  R,<,+, (x\mapsto \lam x)_{\lam \in \R}, \cdot_{\res [0,1]^2}\ra$, every definable smooth (that is, infinitely differentiable) function is already definable in \cal R.

We now collect some definitions and state our results. We assume familiarity with the basics of o-minimality, as they can be found, for example, in \cite{VDD1}. A standard reference for semibounded o-minimal structures is \cite{ed-sbd}. The following definition extends the usual notion of a semibounded structure to a general (not necessarily o-minimal) setting.


\begin{defn}\label{def-sbd}
Let $\cal M = \la M, <, +, \dots\ra$ be an expansion of an ordered group. We call $\cal M$ \emph{semibounded} if there is no definable ordered field with domain $M$ whose order agrees with $<$.  
\end{defn}

There is  a number of statements that could be adopted as definitions of a semibounded structure \cal M and which are known to be equivalent in the o-minimal setting  (see  \cite[Fact 1.6]{ed-sbd}). For example, one could require that there are no definable \emph{poles} (that is, definable bijections between bounded and unbounded sets), or that \cal M is an expansion of $\la M, <, +\ra$ by bounded sets.
The latter statement is in fact  the key definition in  \cite{bel}. The equivalence of (suitable versions of) these statements  in a general setting appears to be an open question. 
 Our choice of Definition \ref{def-sbd} in the current setting is due to the fact that it provides a priori the weakest notion (see relevant questions in Section \ref{sec-sbdqns} below). The main focus of the current work is to establish in this setting our results \ref{main1} - \ref{main2} below. 



In \cite{es}, we introduced certain tameness properties for expansions of real closed fields, and here we extend them to the setting of expansions of ordered groups (Definitions \ref{def-DP}-\ref{def-dim}). Recall that an ordered structure \cal R is called \emph{definably complete} if every bounded definable subset of its universe has a supremum.
For any set $X\sub R^n$, we define its \emph{dimension} as the maximum $k$ such that some projection of $X$ to $k$ coordinates has non-empty interior.\\

\textbf{For the rest of this paper, and unless stated otherwise, we fix an o-minimal expansion $\cal R=\la R, <, +, \dots\ra$  of an ordered group, and a definably complete expansion
$\WR=\la \cal R, \dots\ra$ of $\cal R$. By $\cal L$ we denote the language of $\cal R$. By `definable' (respectively, \cal L-definable), we mean definable in $\WR$ (respectively, in \cal R), with parameters. By $P$ we denote a subset of $R$ of dimension $0$.}\\

If \cal R is a real closed field, we call an \cal L-definable set \emph{semialgebraic}. \smallskip

%
%
%

We  fix throughout the paper the following structures over the reals:
\begin{itemize}
  \item $\overline \R=\la \R, <, +, \cdot\ra$, the real  ordered field.
  \item $\R_{vec}=\la \R, <, +, (x\mapsto \lambda x)_{\lambda\in \R}\ra$, the real  ordered vector space over $\R$, and
  \item $\R_{sbd}=\la \R, <, +, (x\mapsto \lambda x)_{\lambda\in \R}, \cdot_{\res [0,1]^2}\ra$, a  semibounded structure.
\end{itemize}
We note that, by \cite{mpp}, $\R_{sbd}$ is the unique structure that lies strictly between $\R_{vec}$ and $\overline \R$ (in terms of their classes of definable sets).




We let $\Lambda(\cal R)$ be the set of all partial $\es$-definable  endomorphisms of $\la R, <, +\ra$. Then \cal R is called \emph{nonlinear} (\cite{lp})
 if it properly expands $\la R, <, +, \Lambda(\cal R)\ra$. By \cite{pest-tri}, $\cal R$ is nonlinear if and only if it defines a real closed field on some bounded interval.\smallskip

We now extend the tameness properties from \cite{es} to the current setting.

\begin{defn}\label{def-chunk}
Let $Y\sub X\sub R^n$ be two sets. We say that $Y$ is an \emph{\cal L-chunk of $X$} if it is an \cal L-definable cell, $\dim Y= \dim X$, and for every $y\in Y$, there is an open box $B\sub R^n$ containing $y$ such that $B\cap X\sub Y$. Equivalently, $Y$ is a relatively open $\cal L$-definable cell contained in $X$ with $\dim Y=\dim X$.
\end{defn}


\begin{defn}\label{def-DP} We say that $\WR$ has the \emph{decomposition property} \textup{(DP)} if for every definable set $X\sub \R^n$,
\begin{enumerate}
  \item[(I)] there is an \cal L-definable family $\{Y_t\}_{t\in R^m}$ of subsets of $R^n$, and a definable set $S\sub R^m$ with $\dim S=0$, such that $X=\bigcup_{t\in S} Y_t$,

  \item[(II)] $X$ contains an \cal L-chunk.
\end{enumerate}
\end{defn}

\begin{defn}\label{def-dim}
We say that $\WR$ has the \emph{dimension property} (DIM) if for every \cal L-definable family $\{X_t\}_{t\in A}$, and definable set $S\subseteq A$ with $\dim S=0$, we have
$$\dim \bigcup_{t\in S} X_t=\max_{t\in S} \dim X_t.$$
\end{defn}



  %

As mentioned earlier, (DP) and (DIM) extend the corresponding properties from \cite{es} to the current setting.\footnote{We take the opportunity to correct two misprints in \cite{es}.  First, in \cite[Definition 1.1]{es},  the phrase `$\cal L$-definable' should be replaced by `semialgebraic' (as the notation $\cal L$ was not defined there). Second, \cite[Definition 1.3]{es} should be the same with Definition \ref{def-dim} here, with $\cal L$ being the language of $\overline \R$.   This is how those definitions are used  in the rest of that paper.}
In \cite{es}, we showed that if \cal R is a real closed field and $\WR$ satisfies (DP) and (DIM), then $\WR$ defines no new smooth functions.
We extend this theorem to the semibounded setting over the reals (Theorem \ref{main2} below), in two steps. First, in Section \ref{sec-sa}, we prove the following result which holds without the assumption of $\WR$ being over the reals. It ensures that $\WR$ defines no new smooth functions that are not semialgebraic.

\begin{theorem}\label{main1}
Let $\cal R$ be a nonlinear reduct of a real closed field, and  $\WR$ an expansion of $\cal R$ satisfying \textup{(DP)} and \textup{(DIM)}. Let $f: X\sub R^n\to R$ be a definable smooth function, with  open semialgebraic domain $X$. Then $f$ is semialgebraic.
\end{theorem}

%

%

Second, in Section \ref{sec-Rdef}, we restrict \cal R over the reals and let $\WR=\la \cal R, P\ra$. Using a result from Friedman-Miller \cite{FM} (Fact \ref{FM} below), we prove the following proposition. Note that here $\cal R$ is any semibounded structure over the reals, not necessarily $\R_{sbd}$.


\begin{prop}\label{sharp} Let $\cal R= \la \R, <, +, \dots\ra$ be an o-minimal semibounded structure. Suppose $\WR= \la \cal R, P\ra$ has  \textup{(DIM)}. Then $\WR$ is semibounded. 
\end{prop}

We can then conclude our main result.

\begin{theorem}\label{main2} Let $\cal R=\R_{sbd}$, and  assume that $\WR=\la \cal R, P\ra$ satisfies \textup{(DP)} and \textup{(DIM)}. Let $f: X\sub \R^n\to \R$ be a definable smooth function, with open $\cal L$-definable  domain $X$. Then $f$ is \cal L-definable.
\end{theorem}
\begin{proof}
By Theorem \ref{main1}, $f$ is semialgebraic. By Proposition \ref{sharp}, $\WR$ is semibounded. In particular, its reduct $\la \cal R, f\ra$ is semibounded. But this reduct is o-minimal, and hence, by \cite[Theorem 1.4]{pet-reals}, $f$ is definable in $\R_{sbd}$.
\end{proof}

\begin{remark} \label{rmk1}$ $
 The following assumptions of Theorem \ref{main2} are necessary:
    \begin{enumerate}
    \item   $\cal R=\R_{sbd}$ (and not any semibounded structure over the reals).  Indeed, let $\cal R$ be the expansion of $\la \R, <, +\ra$ with all restricted analytic functions, and $\WR=\la \cal R, e^{2\pi\Z}\ra$. Similarly to \cite[Example 4.7]{es},  let $f:(0,1)\to \R$ be the function $f(x)= \sin \log (1/x)$. Clearly, $f$ is not definable in the o-minimal \cal R, since its zero set is an infinite discrete set. We show that $f$ is definable in $\WR$. Let $\lambda:(0,1)\to e^{2\pi \Z}$ be the function sending $x$ to the biggest element of $e^{2\pi\Z}$ lower or equal than $x$. For every $x\in (0,1)$, we have
        $$ f(x) =\sin\log(1/\lam(x)\cdot \lam(x)/x)=\sin\log (\lam(x)/x).$$
        But $\lam(x)/x\in [e^{-2\pi},1]$, $\log ([e^{-2\pi},1 ])=[ -2\pi,0]$, and
        the map
        $$(t,x)\to t/x: \bigcup_{t'\in (0,1)} \{t'\} \times [t', e^{2\pi} t'] \to [ e^{-2\pi},1]$$ is definable in \cal R. Hence, $f$ definable in $\WR$.

  \item $f$ is smooth. If not, we can let $\WR=\la \R_{sbd}, 2^\Z\ra$ and  $f_n$ be the function defined in \cite[Remark 4.8(3)]{es}. Then ${f_n}_{\res (0,1)}$ is definable, $\cal C^{n-1}$, not $\cal C^n$.
  \end{enumerate}
\end{remark}

In Section \ref{sec-examples}, we turn to examples $\WR=\la \cal R, P\ra$ (Proposition \ref{prop-examples} below) to which we can apply Theorem \ref{main2}. The archetypical example is that of $\la \R_{sbd}, 2^\Z\ra$, but our work yields more examples. Let us recall a definition.




\begin{definition}[{\cite{MT}}] Let $f:\R\to \R$ be an \cal L-definable bijection, and $f^n$ the $n$-th compositional iterate of $f$. We say that \cal R is \emph{$f$-bounded} if for every \cal L-definable function $g : \R\to \R$, there is $n\in \N$ such that ultimately $g<f^n$.

Let $c\in\R$ and $f$ an \cal L-definable function such that \cal R is $f$-bounded, and such that $(f^n(c))_n$ is growing and unbounded. We call such $(f^n(c))_n$ an {\em iteration sequence}.
 \end{definition}

\begin{prop}\label{prop-examples}
Let $\cal R= \R_{sbd}$, and $\WR$ be any of the following structures:
\begin{enumerate}
  \item   $\la \cal R,P\ra$, where $P$ is an iteration sequence.
  \item $\la\cal R, \alpha^\Z, \cdot_{\res (\alpha^\Z)^2}\ra$, where $1<\alpha\in \R$.
\end{enumerate}
Then every smooth definable map with open $\cal L$-definable domain is $\cal L$-definable.
\end{prop}

We note that if we replaced $\R_{sbd}$ by $\R_{vec}$ in the above examples, the  conclusion of Theorem \ref{main2} also holds, by Hieronymi-Walsberg \cite{HW}.

%

$ $\\
\noindent\emph{Structure of the paper.}
In Section \ref{sec-prel}, we fix some notation and establish  basic properties for semibounded structures. In Section \ref{sec-sa}, we prove Theorem \ref{main1}.  In Section \ref{sec-Rdef}, we prove Proposition \ref{sharp}.  In Section \ref{sec-examples}, we prove Proposition \ref{prop-examples} and conclude with various open questions  about extending our results further.


$ $\\
\noindent\textbf{Acknowledgments.}
We thank   Philipp Hieronymi and Erik Walsberg for motivating some of the topics of this paper. 






\section{Preliminaries}\label{sec-prel}

In this section, we fix some notation and prove basic facts about semibounded structures.
If $A, B\sub R$, we denote $\frac{A}{B}=\{a/b : a\in A, b\in B\}$. If $t\in R$, we write $\frac{A}{t}$ for $\frac{A}{\{t\}}$.
By a $k$-cell, we mean a  cell of dimension $k$.  If $S\sub R^n$ is a set, its closure is denoted by $\cl S$, with sole exception  $\overline \R$, which denotes the real field. By an open box $B\subseteq R^n$, we mean a set of the form
$$B=(a_1, b_1)\times \ldots \times (a_n, b_n),$$
for some $a_i< b_i\in R\cup\{\pm\infty\}$. By an open set we always mean a non-empty open set.
For a set $X\sub \R$, we define the \emph{convex hull} of $X$, denoted by $conv(X)$, as the set
$$conv(X)=\bigcup_{x<y\in X} [x, y].$$

We prove a useful lemma about our properties.

\begin{lemma}\label{lem-dim2} Assume $\WR=\la \cal R, P\ra$ has \textup{(DP)(II)} and \textup{(DIM)}. Then for every definable set $X$,  $\dim(\cl{X})=\dim(X)$.
\end{lemma}
\begin{proof}
Let $X\subseteq R^n$ be a definable set. Towards a contradiction,  assume that $\dim(X)<\dim(\cl{X})$. By projecting onto some coordinates we may assume that $\dim(\cl{X})=n$. Let $B$ be an open box  contained in $\cl{X}$.  By DP(II), there is an $\cal L$-chunk $Y$ of $X\cap B$. Thus $\dim Y=\dim X\cap B<n$. By definition of an $\cal L$-chunk, there is an open box $B'\subseteq B$ such that $B'\cap X\subseteq Y$. Since $Y$ is \cal L-definable and has dimension $k'<n$, $Y$ is not dense in $B'$. Therefore, $X$ is not dense in $B'$ and thus neither in $B$. Since $B$ is open, $B\nsubseteq \cl{X}$. This is a contradiction and we have the result.
\end{proof}
\begin{rmk} Based on \cite{Mil1}, (DPII) is not necessary to get the conclusion of Lemma \ref{lem-dim2}; see Remark \ref{rmk i-min}.
\end{rmk}

\subsection{Semibounded o-minimal structures}\label{sec-sbd}

In this subsection, we assume that $\cal R$ is a semibounded structure.
Following \cite{pet-sbd, el-sbd}, we say that an interval is \emph{short} if it is possible to define a  field structure on it.   We say that a set is \emph{short} if it is  contained in a product of short intervals. 

\begin{definition}
Let $\cal Y=\{X_t\,:\; t\in A\}$ be an \cal L-definable family. We define the equivalence relation $\sim_{\cal Y}$ as follows:
$$t\sim_{\cal Y} t' \,\,\Lrarr\,\, X_t=X_{t'}.$$
\end{definition}
\begin{lemma} \label{lm com} Let $I\sub M^n$ be a short set, and $\cal Y=\{X_t\}_{t\in A}$ an \cal L-definable family of subsets of $I$.  Then there is a short set $A'\sub A$ of representatives for $\sim_{\cal Y}$.
\end{lemma}
\begin{proof}

We prove the lemma by induction on $n$.  Let $n=1$. By cell decomposition in o-minimal structures, it is easy to see that we may assume that either every $X_t$ is a singleton or every $X_t$ is an open interval.
Suppose every $X_t$ is a singleton, $X_t=\{x_t\}$. Define the map $f:A\to I$, with $t\mapsto x_t$. By the usual dimension properties, as in \cite{el-sbd}, there is a short $A'\sub A$, such that $f(A')=f(A)$, as needed. Suppose now that every $X_t$ is an open interval, $X_t=(a_t, b_t)$. Define $f:A\to I^2$, with $t\mapsto (a_t, b_t)$. Again, there is a short $A'\sub A$, such that $f(A')=f(A)$, as needed. 

Now let $n>1$.
By inductive hypothesis for $\cal C=\{\pi(X_t)\}_{t\in A}$, there is a short set of representatives $C\sub A$ for $\sim_{\cal C}$. For every $s\in C$, consider the set $$Y_s=\{t\in A: \pi(X_t)=\pi(X_s)\}$$ and the family $\cal D_s=\{X_t\}_{t\in Y_s}$. It is enough to show that for every $s\in C$, the statement holds for $\cal D_s$. Namely, it is enough to find a short set of representatives $D_s\sub Y_s$ for $\sim_{\cal D_s}$. Indeed, in that case, $\bigcup_{s\in C} D_s$ will be a set of representatives for $\sim$, and, moreover, by \cite[Lemma 3.6]{el-sbd}, it will be short.

So fix $s\in C$. The family $\cal D_s$ consists of all sets $X_t$, $t\in A$, with $\pi(X_t)=\pi(X_s)$. For every $x\in \pi(X_s)$, consider the set of fibers $\cal F_x=\{(X_t)_x\}_{t\in D_s}$. By the case of $n=1$, there is a short set of representatives $F_ x$ for $\sim_{\cal F_x}$. Then the set $\bigcup_{x\in \pi(X_s)} F_x$  is a set of representatives for $\sim_{\cal D_s}$, again short by \cite[Lemma 3.6]{el-sbd}, as needed.
\end{proof}




In Section \ref{sec-Rdef}, we will use the following fact.

\begin{fact}\label{str function} Let $\cal R=\R_{sbd}$, and $f:X\subseteq \R^n\rightarrow \R$ be an $\cal L$-definable function. Then there is an interval $B\subseteq \R$ and an affine function $\lambda: \R^n\rightarrow \R$, $x\mapsto \sum_i\lambda_ix_i+b$, such that for every $x\in X$, $f(x)\in \lambda(x)+B$.
\end{fact}
\begin{proof}
Easy to see, using \cite[Fact 1.6]{ed-sbd}. 
\end{proof}

\subsection{Open questions}\label{sec-sbdqns}

We conclude this section with some open questions.

\begin{question}\label{Q pole}
Let $\cal M = \la M, <, +, \dots\ra$ be an expansion of an ordered group. Are the following equivalent?
\begin{itemize}
  \item \cal M is semibounded,
  \item \cal M has no definable poles.
\end{itemize}
\end{question}
A potential counterexample to the above question could  be given by the following structure. For $t\in 2^{-\N}$, let $f_t:[t,2t)\to (1/2t,1/t]$ be a linear homeomorphism, and define
$$\cal M=\la\R,<,+,\{f_t:\; t\in 2^{-\N} \}\ra.$$
It is easy to see that \cal M defines a pole, but we do not know if it is semibounded.


\begin{question}\label{qn-sbd}
Let $\cal B$ be the collection of all bounded  sets definable in $\la \overline \R, 2^\Z\ra$. Do
$$\la \R, <, +, \cdot_{\res [0,1]^2}, 2^\Z\ra\text{ and } \la \R, <, +, \{B\}_{B\in \cal B}, 2^\Z\ra$$
have the same definable sets?
\end{question}

As mentioned in the introduction, $\R_{sbd}$ is the unique structure strictly between $\R_{vec}$ and $\overline \R$.

 \begin{question}
 What are the possible structures between $\la \R_{vec}, 2^\Z\ra$ and $\la \overline \R, 2^\Z\ra$?
 \end{question}

Unlike the o-minimal case, there are more than one such structures: besides $\la \R_{sbd}, 2^\Z\ra$, one can consider, for example,  $\la \R_{vec},\cdot_{\res 2^\Z\times\R}\ra$. Similar examples were studied by Delon in \cite{Delon}.

\section{No new non-semialgebraic smooth functions}\label{sec-sa}

In this section, \cal R denotes a nonlinear reduct of a real closed field $\cal R'$, and $\WR$ an expansion \cal R, as fixed in the introduction.
The goal of this section is to show that if $\WR$ has (DP) and (DIM), then every definable smooth function $f:X\sub R^n\to R$ with open semialgebraic domain $X$ is semialgebraic (Theorem \ref{k4} below). The proof  is done in two steps, the first being when $X$ is short. This case is handled by reduction to the semialgebraic case, namely to \cite[Theorem 1.4]{es}.
In order to do this reduction, we first prove some additional lemmas for semibounded structures in Section \ref{sec-sbd2} below.
The general case is done by reduction to the short case, using some basic facts from real algebraic geometry, which we recall in Section \ref{sec-rag}.

\subsection{More on semibounded structures} \label{sec-sbd2}

For the rest of Section \ref{sec-sa}, we fix a short  interval $I =(-a, a)\sub R$ and the  order-preserving semialgebraic diffeomorphism $\tau(x)= \frac{ax}{\sqrt{x^2+1}}: R\rightarrow I$.  We let $\cal I=\la I, <, \oplus,\odot\ra$ be the field structure induced on $I$ from $\cal R'$ via $\tau$. Namely, for every $x,y\in I$,
$$x\oplus y =\tau(\tau^{-1}(x)+\tau^{-1}(y)) $$
and
$$x\odot y=\tau(\tau^{-1}(x)\cdot\tau^{-1}(y)).$$
Denote by $\cal L_\CI$ the language of \cal I. Clearly, \cal I is a real closed field. It is in fact pure.

\begin{fact}[{\cite[Corollary 3.6]{mpp}}]\label{pet}
If $X\subseteq I^n$ is semialgebraic (that is, definable in $\cal R'$), then $X$ is $\cal L_\CI$-definable.
\end{fact}
\begin{proof}
By \cite{opp}, there is a semialgebraic isomorphism $\sigma: \cal R'\to \cal I$. Now let $X\sub I^n$ be semialgebraic. Hence, $\sigma^{-1}(X)$ is semialgebraic. But since $\sigma$ is an isomorphism between the structures $\cal R'$ and \cal I, this means that $X$ is definable in \cal I.
\end{proof}

Definable completeness of \cal R easily implies that \cal I is also definably complete. We write $(\frac{x}{y})_{\cal I}$ for the division  in $\cal I$.  Since the order-topology on $I$ coincides with the subspace topology from $R$,  the dimension of a subset of $I^n$ with respect to either structure is the same. Moreover, if $f:X\sub I^n\to I$ is any function, then continuity of $f$ is invariant between the two structures; that is, $f$ is continuous with respect to $\cal I$ if and only if it is continuous with respect to $\cal R$.
We next prove that smoothness of $f$ is also invariant between the two structures. Let us write  $f\in \cal C^\infty(\cal R)$ if $f$ is smooth in the sense of \cal R (or, rather $\cal R'$),
and $f\in \cal C^\infty(\cal I)$ if it is smooth in the sense of $\cal I$. For $n=1$, we denote
$$ \left(\frac{d f}{d x}\right)_{\cal I}=\lim_{h\rightarrow 0}\left(\frac{f(x\oplus t)\ominus f(x)}{t}\right)_{\cal I}.$$

\begin{lemma}\label{lem-Cinfty} Let $f:X\sub I^n\to I$ be any function with open domain.
Then $f\in \cal C^\infty(\cal R)$ if and only if $f\in \cal C^\infty (\cal I)$.
\end{lemma}
\begin{proof}
We only prove the left-to-right direction, since the other direction is similar. Assume $f\in \cal C^\infty(\cal R)$. Working inductively on the $m$-th partial derivatives of $f$ with respect to \cal I, it is enough to show that each partial derivative of $f$ with respect to \cal I is in $\cal C^{\infty}(\cal R)$. For this, it is enough to show that if $n=1$ and $f\in \cal C^{\infty}(\cal R)$, then there is an $\cal L$-definable function $g\in \cal C^\infty(\cal R)$ such that $g=\left(\frac{df}{dx}\right)_{\cal I}$. We have:

\begin{align*}
 \left(\frac{d f}{d x}\right)_{\cal I}
&=\lim_{t\rightarrow 0}\left(\frac{f(x\oplus t)\ominus f(x)}{t}\right)_{\cal I} \\&
=\lim_{t\rightarrow 0}\left(\frac{ \tau ( \tau^{-1} f \tau (\tau^{-1}(x)+\tau^{-1}(t)) - \tau^{-1} f(x))}{t}\right)_{\cal I} \\ &
=\lim_{t\rightarrow 0} \tau \left(\frac{\tau^{-1} \tau ( \tau^{-1} f \tau (\tau^{-1}(x)+\tau^{-1}(t)) - \tau^{-1} f(x))}{\tau^{-1}(t)}\right)
\\ &
=\lim_{t\rightarrow 0} \tau \left(\frac{ \tau^{-1} f \tau (\tau^{-1}(x)+\tau^{-1}(t)) - \tau^{-1} f \tau (\tau^{-1} x)}{\tau^{-1}(t)}\right)
& \\\\ &\hspace*{-2.2cm}\text{Letting $F=\tau^{-1} f \tau$,
$X=\tau^{-1}(x)$ and $T=\tau^{-1}(t)$, the above equals} \\\\&
=\lim_{T\rightarrow 0} \tau \left(\frac{F(X+T)- F(X)}{T}\right)\\&
=\tau \frac {d F}{dX} \in \cal C^{\infty}(\cal R),
\end{align*}
as needed.
\end{proof}

For the rest of this section, we  fix $\WI$ to be the structure on $I$ induced from $\WR$. Namely,
$$\WI =\la  I, \{X\}_{X \sub I^n \text{ definable}}\ra.$$
Clearly, $\WI$ expands $\cal I$.

\begin{lemma}\label{lem-dpdim} Suppose $\WR$ has \textup{(DP)} and \textup{(DIM)}. Then so does $\WI$.
\end{lemma}
\begin{proof}
Let $X\sub R^n$ be a set definable in $\WI$. \smallskip

\noindent (DP)(I): Observe that $X$ is also definable in $\WR$. By (DPI) for $\WR$  there is an $\cal L$-definable family $\{Y_t\}_{t\in R^m}$ of subsets of $R^n$, and a definable set $S\sub R^m$ with $\dim S=0$, such that $X=\bigcup_{t\in S} Y_t$. By Lemma \ref{lm com}, we may assume that $S\subseteq I^m$. By Fact \ref{pet}, the family $\{Y_t\}_{t\in S}$ is $\cal L_\CI$-definable, as needed.\smallskip

\noindent (DP)(II): Let $X$ be a set definable in $\WI$  and  $Y$  an $\cal L$-chunk of $X$. Since the topologies on $\WI$ and $\WR$ coincide, and, by Fact \ref{pet}, $Y$ is $\cal L_\CI$-definable, it follows that $Y$ is also an $\cal L_\CI$-chunk of $X$.\smallskip

\noindent (DIM): Straightforward.
\end{proof}

\subsection{Real algebraic geometry}\label{sec-rag} Let $\cal R=\la R, <, +, \cdot\ra$ be a real closed field. By an \emph{algebraic set $A\sub R^n$}, we mean the zero set of a polynomial in $R[X]$.  The \emph{Zariski closure} of a set $V\sub R^n$ is the intersection of every algebraic set containing $V$,  denoted by  $\overline V^{zar}$. Note that $\overline V^{zar}$ is algebraic, because $R[X_1,\ldots, X_n]$ is Noetherian.

Let $V$ be an algebraic set. We say that $V$ is \emph{irreducible} if, whenever $V=V_1\cup V_2$, with each $V_i$ algebraic, we have $V=V_i$, for $i=1$ or $2$.

\begin{fact}\label{Coste1}
  Let $X$ be a semialgebraic set. Then $\dim X = \dim (\overline{X}^{zar})$.
\end{fact}
\begin{fact}\label{Coste2} Let $Y$ and $Y'$ be two irreducible algebraic sets of dimension $n$, with $\dim(Y\cap Y')=n$. Then $Y=Y'$. \end{fact}
\begin{proof}
  By \cite[Lemma 3.4]{es}, $Y=Y\cap Y' = Y'$.
\end{proof}

\begin{defn}(\cite[Definitions 2.9.3, 2.9.9]{BCR})\label{def-Nashsub}
A \emph{Nash function} $f: X\sub R^n\to R^m$ is a semialgebraic smooth function with open domain. A \emph{Nash-diffeomorphism}  $f:X\to Y$ is a Nash function which is a bijection and whose inverse is also Nash.

A semialgebraic set $V\subseteq R^m$ is a \emph{Nash-submanifold of dimension $d$} if,  for every $x\in V$, there is a Nash-diffeomorphism $\phi$ from an open semialgebraic neighborhood $U$ of the origin in $R^m$ onto an open semialgebraic neighborhood  $U'$ of $x$ in $R^m$, such that $\phi(0)=x$ and $\phi((R^d\times \{0\})\cap U)=V\cap U'$.
\end{defn}

Note  that the graph of a Nash function with connected domain is a connected Nash-submanifold.

\begin{fact}[{\cite[Lemma 8.4.1]{BCR}}]\label{Coste3}
Let $V\subseteq R^m$ be a connected  Nash-submanifold. Then $\overline V^{zar}$ is irreducible.
\end{fact}

\subsection{Proof of Theorem \ref{main1}} We are now ready to prove the main result of this section.

\begin{theorem}\label{k4} Assume  $\WR$ satisfies \textup{(DP)}, \textup{(DIM)} and is definably complete. Let $f:X\subseteq R^n\rightarrow R$ be a definable smooth function, where $X$ is an open  semialgebraic set. Then $f$ is semialgebraic.
\end{theorem}
\begin{proof}
We proceed in two steps:

\vskip.2cm
\noindent\textbf{Step I.} $\Gamma_f$ is short. We handle this case by reduction to the semialgebraic case, \cite[Theorem 1.4]{es}. First, we claim that we may assume that $\Gamma_f\sub I^{n+1}$. Indeed, after translating, we may assume that $\Gamma_f\sub J^{n+1}$, where $J$ is a short interval. Since any two short intervals are in $\cal L$-definable bijection, there is an $\cal L$-definable bijection that embeds $\Gamma_f$ into $I^{n+1}$.

 We may thus assume that $\Gamma_f\sub I^{n+1}$. In particular, $\Gamma_f$ is definable in $\WI$. Also, since $f\in \cal C^{\infty}(\cal R)$, by Lemma \ref{lem-Cinfty} we obtain $f\in \cal C^{\infty}(\cal I)$.  Now, by Lemma \ref{lem-dpdim}, $\WI$ has (DP) and (DIM), and hence, by  \cite[Theorem 1.4]{es}, $f$ is $\cal I$-definable. In particular, it is semialgebraic, as needed.\vskip.2cm

\noindent\textbf{Step II.} General case.  By \cite{eep}, every open semialgebraic set is a finite union of open cells. Hence we may assume $X$ is an open cell.  Let $\pi$ be the projection on the first $n$-coordinates,  $B$ a short open box that intersects $\Gamma_f$, and $B'\subseteq\pi(B\cap \Gamma_f)$ an open box. Denote $g=f_{\res B'}$. Clearly, $g$ is contained in a short set and  by Step I, $g$ is semialgebraic, and hence Nash. Therefore $\Gamma_g$ is a connected Nash-submanifold. By Fact \ref{Coste3}, the set $Y=\overline{\Gamma_g}^{zar}$ is irreducible, and by Fact \ref{Coste1} it has dimension $n$. \vskip.2cm

\noindent\textbf{Claim.} $\Gamma_f\sub Y$.
\begin{proof}[Proof of Claim]
Let
$$Z=\{x\in X : (x, f(x))\in Y\}.$$
It is enough to show $X\sub Z$.
Note that  $B'\sub Z$, and hence $\dim Z=n$. Assume towards a contradiction that $X\not\subseteq Z$. Since $X$ is connected and open, there is  $z\in \fr(Z)$ and an open short box $(z, f(z))\in D_1$, such that for $D:=\pi(D_1)$, we have
\begin{enumerate}
  \item  $\dim (D\cap Z)=n$, and
  \item $D\sm Z \ne\es$.
\end{enumerate}
Clearly, $\Gamma_{f_{\res D}}$ is short, and hence by Step I, we have that $f_{\res D}$ is semialgebraic. Since $\dim D=n$, the Zariski closure $Y'=\overline{\Gamma_{f_{\res D}}}^{zar}$ has dimension $n$ (Fact \ref{Coste1}). Moreover, the intersection $Y\cap Y'$ contains $\Gamma_{f_{\res D\cap Z}}$ and hence by (2) also has dimension $n$. By Fact \ref{Coste2}, $Y=Y'$. It follows that for every $d\in D$,
$$(d, f(d))\in \Gamma_{f_{\res D}}\sub Y'=Y.$$
This implies $D\sub Z$, which contradicts (3).
\end{proof}
Since $X$ is an $n$-cell, $\Gamma_f\sub Y$ and $\dim Y=n$, by \cite[Lemma 3.10]{es}, we obtain that $f$ is semialgebraic.
\end{proof}

\section{Staying semibounded}\label{sec-Rdef}

In this section, $\cal R=\la \R, <, +, \dots\ra$ denotes a semibounded o-minimal structure over the reals. Besides  reducts of the real field, examples include the expansion of $\la \R, <, +\ra$ by all restricted analytic functions, and others.

Our goal in this section is to prove Proposition \ref{sharp}. We show a slightly stronger version, Proposition \ref{no mult} below. We will need some machinery from \cite{FM}.
For $\WR=\la\cal R,P\ra$, we denote by $\WR^c$ the expansion of $\WR$ by sets for any subsets of $P^k$ for any $k\in \N$.

\begin{definition}[\cite{FM}]
We say that a set $Q\subseteq \R$ is \emph{sparse} if for every $\cal L$-definable function $f: \R^k\rightarrow \R$, $\dim f(Q^k)=0$.
\end{definition}
\begin{lemma}\label{sparse}
 If $\WR=\la\cal R,P\ra$ has \textup{(DIM)}, then $P$ is sparse.
\end{lemma}
\begin{proof}
By (DIM), $$\dim f(P^k) =\dim \bigcup_{t\in P^k}\{f(t)\}=\max_{t\in P^k}\dim \{f(t)\}=0,$$
as required.
\end{proof}


\begin{fact}[{\cite[Last claim in the proof of Theorem A]{FM}}]\label{FM} Assume $P\sub \R$ is sparse.
Let $A\subseteq \R^{n+1}$ be definable in $\WR^\#$  such that for every $x\in \R^n$, $A_x$ has no interior. Then there is an $\cal L$-definable function $f:\R^{m+n}\rightarrow \R$ such that for every $x\in \R^n$, $$A_x\subseteq \cl{f(P^m\times\{x\})}.$$
\end{fact}

Before proving our result, we first need a lemma.

\begin{lemma}\label{R dim}
If $\WR=\la\cal R,P\ra$ has \textup{(DIM)} then so does $\WR^\#$.
\end{lemma}
\begin{proof} Let $\{X_t:\, t\in A\}$ be an \cal L-definable family and let $S\subseteq A$ be a set of dimension $0$ definable in $\WR^\#$. For simplicity, we may assume that all the $X_t$'s have the same dimension $n$. By Lemma \ref{sparse}, $P$ is sparse and by Fact \ref{FM}, there is an \cal L-definable function $f:\R^k\to \R^n$ so that $S\subseteq \cl{f(P^k)}.$
We show that $\cl{f(P^k)}$ has dimension $0$. For contradiction, suppose that $\cl{f(P^k)}$ has interior. Without loss of generality, we may assume that $\cl{f(P^k)}$ is an open interval $I$.  Thus $f(P^k)$ is dense in $I$. Moreover, by (DIM), $f(P^k)$ has dimension $0$. Let $S_1=\{0\}\times f(P^k)$ and $S_2= I\sm f(P^k)\times \{0\}$. It is easy to see that $S_1\cup S_2$ has dimension $0$, and hence by (DIM) for $\WR$,
$$\dim\left(\bigcup_{x\in S_1\cup S_2}\{\pi_1(x),\pi_2(x)\}\right)=0.$$
But $\bigcup_{x\in S_1\cup S_2}\{\pi_1(x),\pi_2(x)\}= I,$
a contradiction. Hence $\dim(\cl{f(P^k)})=0$.
\par Since $\WR$ has (DIM), $$\dim\left(\bigcup_{t\in S}X_t\right)\leq\dim\left(\bigcup_{t\in \cl{f(P^k)}\cap A}X_t\right)=\max_t\dim(X_t)=n$$ and we have the result.
\end{proof}
\begin{rmk}\label{rmk i-min}
Note that the above proof shows that (DIM) implies that every definable subset of $\R$ has interior or is nowhere dense. The latter condition (also known as  ``i-minimality" by Fornasiero) is shown in Miller (see \cite{Mil1}) to imply that for every definable set $X$, $\dim(X)=\dim(\cl{X})$.
\end{rmk}

We are now ready to prove our result.
\begin{proposition}\label{no mult}
If $\WR=\la\cal R,P\ra$ has \textup{(DIM)},
then $\WR^\#$ is semibounded.
\end{proposition}
\begin{proof}
\par If $P$ is finite, then $\WR$ is semibounded and o-minimal and the result is clear.

\par Assume towards a contradiction that there is a field $I=\la\R,<,+_I,\cdot_I,0_I,1_I\ra$ definable in $\WR^\#$. Note that  $\cdot_I$ could be different from the standard multiplication. For simplicity, we assume that $0_I=0$ and  use the standard multiplication and division notations (the addition used in the proof being only the standard one).  Note that since the order of $I$ is the standard one, for every $x\in \R$, $\lim_{t\to \infty}x/t=0$.
\par  Since  there is a pole definable in $\WR^\#$, and $P$ is infinite, by taking  either $P$, if it is unbounded, or a homeomorphic image of $P$  via some bijection $f:conv(P)\to \R$, there is a definable  unbounded $0$-dimensional set $S$. We consider the family $\{xS\,:\; x\in \R\}$. By  Lemma \ref{sparse} and Fact \ref{FM}, there is an $\cal L$-definable function $f:\R^{k+1}\rightarrow \R$ such that
\begin{equation}\label{eq-xS}
  xS\subseteq \cl{f_x(P^k)}.
\end{equation} By Fact \ref{str function}, there is a bounded set $B$ and  linear functions $\lambda:\R^k\rightarrow \R$ and $b:\R\to \R$, such that
\begin{equation}
\label{eq-fxP} f(x, P^k)\subseteq B+\lambda P^k +b(x).
\end{equation}
We prove that
$$\R\sub \overline{\left(\frac{\lam P^k}{S}\right)},$$
which will contradict Lemma \ref{lem-dim2}. To see this, let $x\in \R$ and $\ve>0$. We show that there is $p\in P^k$ and $t\in S$, with
$$\left|x-\frac{\lam p}{t}\right|\le\ve.$$
Take $t\in S$ with $\frac{B+b(x)}{t}<\ve$. By (\ref{eq-xS}) and (\ref{eq-fxP}), there is $p\in P^k$, with $$x\in \overline{\left( \frac{\lam p +B+b(x)}{t}\right)} = \left[\frac{\lam p}{p}- \ve, \frac{\lam p}{p}+\ve\right],$$
as required.

By (\ref{eq-xS}) and (\ref{eq-fxP}), $$x\in Z=\cl{\{y/t\\,:\; y\in \lambda(P^k),\,t\in S\}}.$$
Therefore $Z$ has interior. Moreover, by Lemma \ref{R dim}, $\WR^\#$ has (DIM) and  hence $$\dim(Z)=\dim\big(\{y/t\,:\; y\in \lambda(P^k),t\in S\}\big)=\max_{(y,t)\in P^k\times S}\dim\big( \{\lambda(y)/t\}\big)=0.$$
This is a contradiction.
\end{proof}

\begin{question}
Is it true that if $\la \cal R, P\ra$ has  \textup{(DP)}, then so does $\la \cal R, P\ra^\#$?
\end{question}

 \begin{question}
 Is Theorem \ref{main2} true for $f$ definable in $\WR^\#$?
 \end{question}


\section{Examples}\label{sec-examples}

Throughout this section, $\cal R=\la \R, <, +, \dots\ra$ denotes an o-minimal semibounded structure over the reals. 
 Our goal is to prove Proposition  \ref{prop-examples}. For (1), our approach is the following. First, we show that  under a certain quantifier elimination result, (DP)(I) holds (Proposition \ref{DP1}). Together with d-minimality and the following lemma, we can then conclude  its proof. For (2), we reduce the statement to that of $\la \overline \R, \alpha^\Z\ra$ from \cite{es}, using Proposition \ref{sharp}.

  Recall (\cite{Mil1}) that $\WR$ is called \emph{d-minimal} if every definable subset of $R$ is the union of an open set and finitely many discrete sets,  uniformly in parameters.



\begin{lemma}\label{d-min}
Suppose $\WR=\la \cal R,P\ra $ is d-minimal and has \textup{(DP)(I)}. Then it has \textup{(DP)(II)} and \textup{(DIM)}.
\end{lemma}
\begin{proof}
By  \cite[Proposition 4.15]{es}, we have DP(II). By \cite[Remark 1.5(1)]{es}, we have (DIM). (In that reference $\WR$ expanded a field, but the proof of Remark 1.5(1) did not use that assumption.)
\end{proof}

\subsection{(DP)(I)}\label{sec-DPI}

In what follows, we assume that $P$ is discrete, closed in its convex hull, has no maximal element, and $0<P$. We define $\lambda:\R\rightarrow P\cup \{0\}$,
$$\lambda(x)=\begin{cases}
  \max(P\cap (-\infty,x]) & \text{if $x\in conv(P)$}\\
  0 & \text{otherwise}
\end{cases}
$$
(which exists since $P$ is a discrete set, closed in its convex hull). We define $s:P\rightarrow P$ to be the \emph{successor} function in $P$; namely, 
$$s(x)= \min\{y\in P: x<y\}.$$

By \emph{basic functions} we mean $\lambda$, $s$, $s^{-1}$ and all $\cal L$-definable functions.

\begin{proposition}\label{DP1} 
Assume that every definable  set $X\subseteq \R^l$ is a finite union of sets $Y$, each satisfying the following property:

$(A)$: there are definable functions $f_1,\ldots, f_n:\R^l\rightarrow \R$ and $g_1,\ldots g_m:\R^l\rightarrow \R$, which are given by compositional iterates of basic functions, such that
$$Y=\{x\in \R^l\,:\;\forall i, j,\,  f_i(x)=0,\; g_j(x)>0\}.$$

 Then $\la \cal R,P\ra$ has \textup{DP(I)}.
\end{proposition}
\begin{proof} We begin with a claim.\\

\noindent\textbf{Claim.} {\em
Let $h$ be a composition of basic functions. Then there is an $\cal L$-definable function $f$ and a definable set $S\subseteq P^k$,  such that for all $x\in \dom(h)$,}
\[\tag{*} h(x)=z \text{  if and only if  there is $y\in S$ such that $ f(x,y)=z$}.
\]
\begin{proof}[Proof of Claim]
By induction on the number of iterations of basic functions  which compose $h$. 

For $h=\lambda$, let $S=\Gamma_s\subseteq P^2$ and $f(x,y_1,y_2)=y_1$ if $y_1\leq x<y_2$, and not defined otherwise. We verify $(*)$.  If $\lambda(x)=z$ then $f(x,z,s(z))=z$. By definition of $f$, if $f(x,y_1,y_2)=y_1$, then $y_1\leq x<y_2$, and since $(y_1,y_2)\in \Gamma_s$, we have $y_2=s(y_1)$ and $\lambda(x)=y_1$.  Furthermore, we see that if there is $y \in S$ such that $f(x,y) $ is defined then $f(x,y)=h(x)$. The cases  $h=s,\, s^{-1}$ are similar and the case where $h$ is $\cal L$-definable is straightforward.

Now let $h=h_{n+1}(h_1,\ldots, h_k)$ where $h_{n+1}$ is a basic function and assume that  for $1\le j\leq k$, there are some $\cal L$-definable functions  $h'_j(x,y)$   and definable $S_j\subseteq P^{k_j}$ such that for all $x\in \dom(h)$, $$h_j(x)=z \text{ if and only if  there is $y\in S_j$ such that }h_j(x)=h'_j(x,y).$$
For $h_{n+1}=\lambda$ (thus $k=1$), we define $f$ exactly similarly to the last paragraph, namely $f(x,a_1,a_2,y)=a_1$ if $h'_1(x,y)$ is defined and  $a_1\leq h'_1(x,y)<a_2$, and not defined  otherwise. We verify $(*)$. If $h(x)=a_1$ then there are $(a_1,a_2)\in \Gamma_s, y_1\in S_1$ such that $a_1\leq h_1'(x,y_1)<a_2$ and $h_1'(x,y_1)=h_1(x)$. Thus $f(x,a_1,a_2,y_1)=a_1$.   If there is $y\in S_1$ such that $h_1'(x,y)$ is defined then $h_1'(x,y)=h_1(x)$ and if there are $(a_1,a_2)\in \Gamma_s$  such that $f(x,y,a_1,a_2)$ is defined (that is, $a_1\leq h'_1(x,y)<a_2$) then $ h(x)=a_1=f(x,y,a_1,a_2).$
\par  Again, the cases $h_{n+1}=s,s^{-1} $ are similar and the case $h_{n+1}$ $\cal L$-definable is straightforward.
\end{proof}

\smallskip Now let $X$ be a definable set. By hypothesis, there are $f_1,\ldots,f_k$ and $g_1,\ldots,g_{k'}$, which are compositional iterates of basic functions, such that $$X=\{x\in \R^l\,:\;\forall i, j,\,  f_i(x)=0,\; g_j(x)>0\}$$
Let $f'_i$,   $S_i$ the maps and sets of dimension $0$ given by the claim for $h=f_i$, and $g_j'$, $K_j$, for $h=g_j$. That is, for every $i,j$, we have that
$$f_i^{-1}(0)=\bigcup_{t\in S_i}f_i'(-,t)^{-1}(0),$$
$$g_j^{-1}(\R^{>0})=\bigcup_{t\in K_j}g_j'(-,t)^{-1}(\R^{>0}).$$
Note that $X$ has the form $$\bigcap_{s\leq m}\bigcup_{t\in S_s}Y_{s,t}$$where $m=k+k'$, for every $s\leq m$, $t\in S_s$, $ Y_{s,t}$ is an \cal L-definable set. 

To prove that $X$ has the form $\bigcup_{t\in S}X_t$ where $\{X_t\,:\; t\in S\}$ is a small subfamily of an \cal L-definable  family of sets, by an easy induction it is sufficient to prove that the intersection of two sets of the form $\bigcup_{t\in S'}Y_t$ where there  is an \cal L-definable family $\{Y_t\,:\; t\in A\}$
and $S'\subseteq A$ has dimension $0$ is itself a set of this form. Let $X_1=\bigcup_{t\in S_1} X_{1,t}$ and $X_2=\bigcup_{t\in S_2}X_{2,t}$ where there are two \cal L-definable families $\{X_{i,t}\,:\;t\in A_i\}$
 and $S_i\subseteq A_i$ being of dimension $0$. Then $$X_1\cap X_2 =\bigcup_{(t_1,t_2)\in S_1\times S_2}X_{1,t_1}\cap X_{2,t_2}$$
and the family $\{X_{1,t_1}\cap X_{2,t_2}\,:\; t_1\in S_1,\,t_2\in S_2\}$ is a small subfamily of the \cal L-definable family $Z=\{X_{1,t_1}\cap X_{2,t_2}\,:\; t_1\in A_1,\,t_2\in A_2\}$. Moreover, by cell decomposition in o-minimal structures, we may assume that $Z$ is a family of cells.
 This proves the result.
\end{proof}

 We are now ready to conclude the main result of this section.

 \begin{proof}[Proof of Proposition \ref{prop-examples}]
   For (1), we prove that it has (DP) and (DIM), and hence Theorem \ref{main2} directly applies. By Proposition \ref{d-min}, it suffices to show that it satisfies (DP)(I) and d-minimality. For (DP)(I), by Proposition \ref{DP1}, we only need to show  Condition (A). Both Condition (A)  and d-minimality are shown in \cite{MT}.

For (2), we cannot directly apply Theorem \ref{main2}, because we do not know if $\WR$ satisfies (DP)(I). However, we can derive the result as follows. Let $f$ be a smooth definable function with open \cal L-definable domain. Observe that $f$ is also definable in $\la \overline \R, \alpha^\Z\ra$. By \cite{es}, $f$ is semialgebraic. Also by \cite{es},  $\la \overline \R, \alpha^\Z\ra$ satisfies  (DIM). Since (DIM) is preserved under taking reducts, $\la \cal R, \alpha^\Z\ra$ also satisfies (DIM). Therefore, by Proposition \ref{no mult}, $\WR^\#$ is semibounded, and hence so is its reduct $\la \cal R, f\ra$.   But this reduct is o-minimal, and hence by \cite[Theorem 1.4]{pet-reals}, $f$ is definable in $\R_{sbd}$.
 \end{proof}

\subsection{Open questions}
 We finish with some natural questions and comments that arise from the current work.

\begin{question}
Do $\la \R_{sbd}, 2^\Z\ra$ and  $\la \R_{sbd}, 2^\Z,\cdot_{\upharpoonright (2^\Z)^2}\ra $ have \textup{(DP)}?
\end{question}

The current examples concern semibounded structures that expand $\R_{sbd}$. Suppose that $\cal R'$ is a structure that lies between $\la \R, <, +\ra$ and $\R_{sbd}$, such as $$\cal R'=\la \R, <, +, \cdot_{\res [0,1]^2}\ra.$$

\begin{question} For which $P$ does  $\la \cal R', P\ra$ satisfy the conclusions of Theorem \ref{main2} and Proposition \ref{prop-examples}? In particular, does $\la \R, <, +, \Z\ra$ do? (The last question was asked by Hieronymi.)
\end{question}

We note here that $\la \R_{sbd}, \Z\ra$, and even $\la \R, <, +, (x\mapsto \sqrt{2} x), \Z\ra$, are not d-minimal, as shown in \cite{Hiero two gps}.

\begin{question}
Let $\WR$ be  $ \la \R_{vec}, P\ra $ or $\la \R_{sbd}, P\ra$, where $P$ is an iteration sequence or $2^\Z$. 
Is the open core of  $\la \WR, \R^{alg}\ra$ equal to $\WR$? (extending Khani's relevant result for $\la\cl{\R},2^\Z,\R^{alg}\ra$ in \cite{khani}).
\end{question}

\end{document}